\documentclass[12pt]{amsart}

\usepackage{graphicx}

\usepackage{amssymb, amsmath, amsthm, enumerate}
\usepackage{longtable}
\usepackage{environ}

\theoremstyle{plain}
\newtheorem{thm}{Theorem}
\newtheorem{lem}{Lemma}
\newtheorem{prb}{Problem}
\newtheorem*{cnj}{Conjecture}
\theoremstyle{definition}

\newcommand{\vll}{\operatorname{vol}}
\newcommand{\eps}{\varepsilon}
\newcommand{\mb}{\mathbf}
\newcommand{\cs}{c.s.\@ }
\newcommand{\hcp}{h.c.p.\@ }

\makeatletter

\newcounter{asyfigcntr}
\@ifundefined{asy}{%
  \NewEnviron{asy}{%
    \stepcounter{asyfigcntr}%
    \includegraphics{\jobname-\theasyfigcntr}%
  }
}{}

\begin{document}

\title{Pessimal packing shapes}

\author{Yoav Kallus}
\address{Yoav Kallus, Santa Fe Institute, 1399 Hyde Park Road, Santa Fe, New Mexico 87501}

\date{\today}

\begin{abstract}We address the question of which convex shapes, when packed as densely as possible under certain restrictions,
fill the least space and leave the most empty space.
In each different dimension and under each different set of restrictions, this question is expected
to have a different answer or perhaps no answer at all. As the problem of identifying global minima in most cases
appears to be beyond current reach, in this paper we focus on local minima. We review some known results
and prove these new results: in two dimensions, the regular heptagon is a local minimum of the double-lattice packing density,
and in three dimensions, the directional derivative (in the sense of Minkowski addition) of the double-lattice packing density
at the point in the space of shapes corresponding to the ball is in every direction positive. 
\end{abstract}

\maketitle

\section{Introduction}

An $n$-dimensional convex body is a convex, compact, subset of $\mathbb{R}^n$ with nonempty interior. 
The space of convex bodies, denoted $\mathcal{K}^n$, can be endowed with
the Hausdorff metric:
$$\operatorname{dist}(K,K') = \min \{\eps:K'\subseteq K_\eps\text{ and } K\subseteq K'_\eps\}\text,$$
where $K_\eps = \{\mb{x}+\mb{y}:\mb{x}\in K, ||\mb{y}||\le\eps\}$ is the $\eps$-parallel body
of $K$.

A set of isometries $\Xi$
is said to be admissible for $K$ if the interiors of $\xi(K)$ and $\xi'(K)$
are disjoint for all distinct $\xi,\xi'\in\Xi$. The {\it (lower) mean volume} of $\Xi$ can be defined
as $d(\Xi) = \liminf_{r\to\infty} (4\pi r^3/3) / |\{\xi\in\Xi : ||\xi(0)||<r\}|$.
The collection $\{\xi(K) : \xi \in \Xi\}$ for an admissible $\Xi$ is called a packing of $K$
and said to be produced by $\Xi$.
Its density is the fraction of space it fills: $\vll(K)/d(\Xi)$.
The packing density of a body $K$, denoted $\delta(K)$ is the supremum of $\vll(K)/d(\Xi)$ over all
admissible sets of isometries. Groemer proves some basic results about packing densities,
including the fact that the supremum is actually achieved by some packing and the fact that
$\delta(K)$ is continuous \cite{GroemerBasic}. Groemer's result apply also to the restricted
packing densities which we define below.

An inversion about a point $\mathbf{x}$ is the isometry $I_\mathbf{x}:\mathbf{y}\mapsto2\mathbf{x}-\mathbf{y}$.
While the group of isometries of $\mathbb{R}^n$ is not preserved under conjugation by 
an affine transformation of $\mathbb{R}^n$, the subgroup made of all translations and inversions about points is invariant.
It is interesting to consider packings
produced by sets of only translations and inversions. The supremum of $\vll(K)/d(\Xi)$ over 
packings restricted in this way, denoted $\delta_{T^*}(K)$, is preserved
under affine transformations of $K$ and therefore we may say that the domain of
$\delta_{T^*}(K)$ is the space of affine equivalence classes of convex bodies.
Macbeath showed that this space (with the induced topology) is compact \cite{Macbeath},
and so $\delta_{T^*}$ achieves a global minimum.

Similarly, we may restrict to packings produced only by a set of translations, produced only
by the set of elements of a group of translations (namely a Bravais lattice, or simply {\it lattice} hence), or produced only by
the set of elements of a group of translations and inversions (namely a {\it double lattice},
after G. Kuperberg and W. Kuperberg \cite{KuperbergDouble}),
and define respectively $\delta_T(K)$, $\delta_L(K)$, and $\delta_{L^*}(K)$ in the obvious way.
By the same argument as for $\delta_{T^*}(K)$, all these functions must also achieve a
global minimum. The following problem has been suggested, for example by A. Bezdek and
W. Kuperberg \cite{BezKup} (but see also Ref. \cite{BraMosPac}):

\begin{prb}\label{TTsLLs}In $n$ dimensions, what are the minima of $\delta_T$, $\delta_{T^*}$, $\delta_L$,
and $\delta_{L^*}$? Which bodies achieve these minima?\end{prb}

F{\'a}ry showed that in two dimensions, triangles are the unique minimum of $\delta_L$ \cite{Fary,Courant}.
Also, due to a result of L. Fejes T{\'o}th, $\delta_L=\delta_T$ in two dimensions,
so triangles also minimize $\delta_T$ \cite{LFT2D,Rogers2D}.

A body $K$ is said to be centrally symmetric (c.s.) if there is a point
$\mathbf{x}$ such that $I_{\mathbf{x}}(K)=K$.
It is reasonable to restrict the functions $\delta_T$ and $\delta_L$
to the space $\mathcal K^n_0$ of \cs bodies and ask for the minima of these restricted functions,
since these bodies correspond to unit balls in finite-dimensional Banach spaces.
Therefore the following question is natural to ask:

\begin{prb}\label{cs}In $n$ dimensions, what are the minima of $\delta_T$ and $\delta_L$
among \cs bodies? Which bodies achieve these minima?\end{prb}

In two dimensions, Reinhardt conjectured that a certain smoothed octagon -- a regular octagon
whose corners are rounded off by arcs of hyperbolas -- minimizes $\delta_L$ \cite{Reinhardt,mahler98}.
Due to the same result of L. Fejes T{\'o}th, we have that $\delta(K)=\delta_T(K)=\delta_L(K)$ for \cs
bodies $K$ in two dimensions \cite{LFT2D}.

\begin{figure}
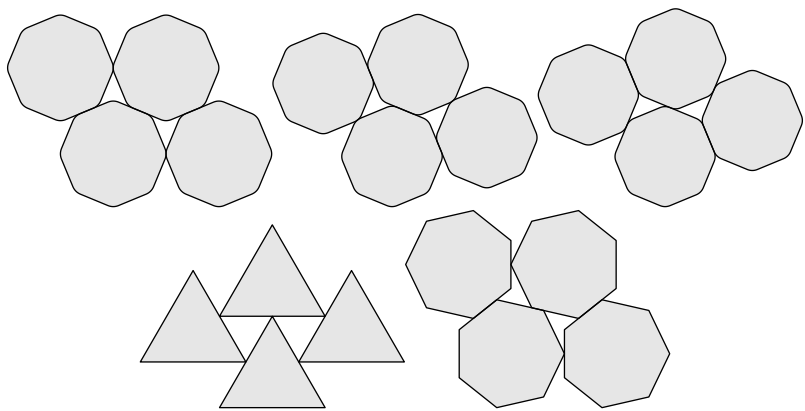
\begin{center}
\begin{asy}
import oct;
size(100);

path rad = (0,0)--(rotate(0)*(4,0));
pair a1 = IP(rad,oct);
real la1 = a1.x*a1.x + a1.y*a1.y;
pair c1 = rotate(90)*a1*d/la1;
path c2 = c1--(c1+a1);
path c3 = c1--(c1-a1);
pair a2 = IP(c2,oct);
pair a3 = IP(c3,oct);

fill(oct,lightgray); fill(shift(2*a1)*oct,lightgray); fill(shift(2*a2)*oct,lightgray); fill(shift(2*a3)*oct,lightgray);
D(oct); D(shift(2*a1)*oct); D(shift(2*a2)*oct); D(shift(2*a3)*oct);
\end{asy}
\begin{asy}
import oct;
size(100);

path rad = (0,0)--(rotate(11.25)*(4,0));
pair a1 = IP(rad,oct);
real la1 = a1.x*a1.x + a1.y*a1.y;
pair c1 = rotate(90)*a1*d/la1;
path c2 = c1--(c1+a1);
path c3 = c1--(c1-a1);
pair a2 = IP(c2,oct);
pair a3 = IP(c3,oct);

fill(oct,lightgray); fill(shift(2*a1)*oct,lightgray); fill(shift(2*a2)*oct,lightgray); fill(shift(2*a3)*oct,lightgray);
D(oct); D(shift(2*a1)*oct); D(shift(2*a2)*oct); D(shift(2*a3)*oct);
\end{asy}
\begin{asy}
import oct;
size(100);

path rad = (0,0)--(rotate(22.5)*(4,0));
pair a1 = IP(rad,oct);
real la1 = a1.x*a1.x + a1.y*a1.y;
pair c1 = rotate(90)*a1*d/la1;
path c2 = c1--(c1+a1);
path c3 = c1--(c1-a1);
pair a2 = IP(c2,oct);
pair a3 = IP(c3,oct);

fill(oct,lightgray); fill(shift(2*a1)*oct,lightgray); fill(shift(2*a2)*oct,lightgray); fill(shift(2*a3)*oct,lightgray);
D(oct); D(shift(2*a1)*oct); D(shift(2*a2)*oct); D(shift(2*a3)*oct);
\end{asy}
\linebreak
\begin{asy}
import cseblack;
import olympiad;
usepackage("amssymb");
size(100);

path T=(0,0)--(1,0)--(1./2.,sqrt(3.)/2.)--cycle;
pair a1 = (3./4.,sqrt(3.)/4.);
pair a2 = (0,sqrt(3.)/2.);
pair a3 = (-3./4.,sqrt(3.)/4.);
fill(T,lightgray); fill(shift(a1)*T,lightgray); fill(shift(a2)*T,lightgray); fill(shift(a3)*T,lightgray);
D(T); D(shift(a1)*T); D(shift(a2)*T); D(shift(a3)*T);
\end{asy}
\begin{asy}
import hept;
size(100);

pair a1 = M0-K0;
pair a2 = 2*P1;
pair a3 = 2*P2;
fill(M,lightgray); fill(shift(a1)*M,lightgray); fill(shift(a2)*rotate(180)*M,lightgray); fill(shift(a3)*rotate(180)*M,lightgray);
D(M); D(shift(a1)*M); D(shift(a2)*rotate(180)*M); D(shift(a3)*rotate(180)*M);
\end{asy}
\caption{\label{packfig}Densest packing structure of pessimal packing shapes. The Reinhardt octagon
has a one-parameter family of optimal lattices, each of which fills $0.90241\ldots$ of the plane,
which is conjectured to be less than is filled by the densest lattice packing of any other
\cs shape. The top row shows three examples from this family. The densest lattice packing of triangles
(bottom left) fills $2/3$ of the plane and is less dense than the densest lattice packing
of any other shape. The densest double-lattice packing of regular heptagons (bottom right)
fills $0.89269\ldots$ of the plane and is conjectured to be less dense than the densest
double-lattice packing of any other shape.}
\end{center}\end{figure}

By contrast to the functions considered in Problems 1--2, $\delta(K)$ is not invariant under
affinities, but only under isometries and dilations.
Therefore, its infimum over all bodies (which is bounded from below by the minimum of $\delta_{T^*}$)
is in theory not necessarily achieved by any particular body. In three dimensions, 
the claim that the ball is the minimum of $\delta(K)$ has come to be known
as Ulam's packing conjecture, due to a remark Gardner attributes to Ulam,
though there is no evidence to confirm that Ulam ever stated it as a conjecture
\cite{Gardner95}. More generally, it is natural to ask,

\begin{prb}\label{delta}In $n$ dimensions, what is the infimum of $\delta$?
Is this infimum achieved by some body?\end{prb}

So far, with the exception of the case of $\delta_T$ and $\delta_L$ in two dimensions
and the trivial case of one dimension, none of the problems 1--3 have been solved in any dimension.
There are two kinds of partial answers that have been successfully obtained:
lower bounds and local minima. In this paper we will focus on the results of the second kind and
content ourselves with a few references to results of the first kind \cite{KuperbergDouble,Smith,Ennola,Hlawka}.
Each of the problems 1--3 lends itself to a local variation: which bodies are a local minimum
of the function in question? In two dimensions, Nazarov showed that Reinhardt's smoothed
octagon is a local minimum of $\delta_L$ (and therefore also $\delta_T$) among \cs bodies \cite{Nazarov}.
In three dimensions, I showed that the ball is
a local minimum of $\delta_L$ among \cs bodies \cite{kalluspack}, and therefore also a local minimum of
$\delta_T$ and $\delta$ among \cs bodies, due to Hales's confirmation of Kepler's
conjecture \cite{HalesKepler}.

In this paper we show that the regular heptagon is a local minimum
of $\delta_{L^*}$. Also, failing to show that the three-dimensional ball is a local minimum
of $\delta$, we show that the directional derivative at the ball with
respect to Minkowski addition is positive in all directions.

\section{The regular heptagon}

Let $K$ be a two-dimensional convex body (hence, {\it domain}) of area $A$. We say that a chord
is an affine diameter of $K$ if it is at least as long as all parallel chords, and we call its length
the length of $K$ in its direction. We say an inscribed parallelogram
is a half-length parallelogram if one pair of sides is half the length of $K$
in the direction parallel to them. G. Kuperberg and W. Kuperberg have shown that in two dimensions $\delta_{L^*}(K)=A/2\Delta(K)$,
where $\Delta(K)$ is the area of the half-length parallelogram of least area inscribed in $K$ \cite{KuperbergDouble}.
They also show that $\delta_{L^*}(K)\ge\sqrt{3}/2=0.86602\ldots$ for all domains $K$ \cite{KuperbergDouble}.
Doheny shows that this bound is not sharp \cite{Doheny}. Here we show that the regular heptagon,
for which $\delta_{L^*}(M)=0.89269\ldots$ (exact value below), is a local minimum. It
is reasonable to conjecture that this is also a global minimum.

For definiteness, let us fix a regular heptagon $M$ with vertices $\mb{m}_i = R^i(1,0)$, $i=0,\ldots,6$
where $R^i$ is a counter-clockwise rotation by $2\pi i/7$ about the origin
(we understand the label $i$ to take values in $\mathbb{Z}/7\mathbb{Z}$).
The coordinates of the vertices are then in the field extension $\mathbb{Q}(u,v)$, where $u=\cos\pi/7$ and $v=\sin\pi/7$,
and we will give all explicit numbers below in the reduced form $a+bu+cu^2+v(d+eu+fu^2)$.
The least-area half-length parallelogram inscribed in $M$ (see Figure \ref{heptfig}) is the rectangle $\mb{p}_1\mb{p}_2\mb{p}_3\mb{p}_4$,
where $\mb{p}_1=(1-a)\mb{m}_1+a\mb{m}_2$, $\mb{p}_2=(1-b)\mb{m}_2+b\mb{m}_3$, $\mb{p}_3=(1-b)\mb{m}_5+b\mb{m}_4$,
$\mb{p}_4=(1-a)\mb{m}_6+a\mb{m}_5$, $a=\tfrac{7}{4}-2u^2$, and $b=-\tfrac{1}{2}+u^2$.
As the area of this rectangle is given by $\Delta=(-19 + 2u +56u^2)v/8$ and the area of the heptagon is given by $A=7uv$,
the double-lattice packing density of $M$ is $A/2\Delta = \tfrac{2}{97}(-111 + 492 u - 356 u^2) = 0.89269\ldots$.
This rectangle, of course is one of seven equivalent rectangles $R^i(\mb{p}_1\mb{p}_2\mb{p}_3\mb{p}_4)$, $i=0,\ldots,6$.

\begin{figure}
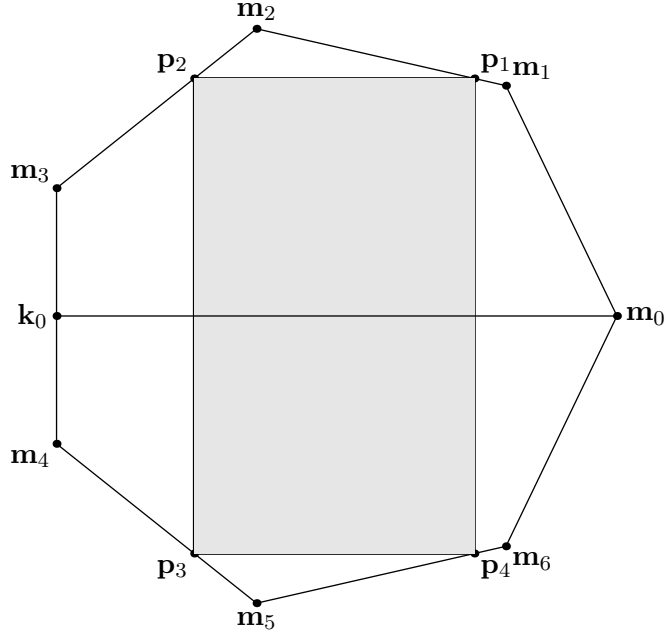
\begin{center}
\begin{asy}
import hept;
size(250);

D("\mathbf{m}_0",(1,0),E);
D("\mathbf{m}_1",rotate(360./7.)*M0,NE);
D("\mathbf{m}_2",rotate(360./7.)*M1,N);
D("\mathbf{m}_3",rotate(360./7.)*M2,NW);
D("\mathbf{m}_4",rotate(360./7.)*M3,SW);
D("\mathbf{m}_5",rotate(360./7.)*M4,S);
D("\mathbf{m}_6",rotate(360./7.)*M5,SE);

D("\mathbf{p}_1",(1-A)*M1 + A*M2,NE);
D("\mathbf{p}_2",(1-B)*M2 + B*M3,NW);
D("\mathbf{p}_3",(1-B)*M5 + B*M4,SW);
D("\mathbf{p}_4",(1-A)*M6 + A*M5,SE);

D("\mathbf{k}_0",M0+2*P2-2*P1,W);

D(M);

path P=D(P1--P2--P3--P4--cycle,linewidth(1));
fill(P,lightgray);
path L=D(M0--K0);
\end{asy}
\caption{\label{heptfig}The half-length parallelogram of least area inscribed in the regular heptagon.}
\end{center}\end{figure}

Let us now consider a different domain $M'$, with area $A'$ and least-area half-length parallelogram
of area $\Delta'$. We will be interested in the limit that $M'$ becomes more and more similar to $M$.
Therefore, let us assume that $(1-\eps)M\subseteq M'\subseteq(1+\eps)M$, and we will explore what happens
as we let $\eps$ approach $0$. We wish to prove that there exists $\eps>0$ such that $A'/2\Delta'\ge A/2\Delta$
for all $M'$. We will prove this in two steps: we first prove that $A'/2\Delta'\ge A/2\Delta$ if $M'$
is also a heptagon, and then we prove that $M'$ is a heptagon if $A'/2\Delta'\le A/2\Delta$.

\begin{thm}\label{heptthm} There exists $\eps>0$ such that if $M'$ is a heptagon and $(1-\eps)M\subseteq M'\subseteq(1+\eps)M$
then $A'/2\Delta'\ge A/2\Delta$, with equality only when $M'$ is affinely equivalent to $M$.\end{thm}

\begin{proof}
Let the vertices of $M'$ be $\mb{m}_i' = R^i (1+x_i,y_i)$, $i=0,\ldots,6$. 
Denote by $\mb{x}$ the vector $(x_0,y_0,x_1,\ldots,x_6,y_6)\in\mathbb{R}^{14}$.
By the affine invariance of the double-lattice packing density, we may assume
without loss of generality that $\mb{x}$ lies, say, in the 8-dimensional subspace $W\subseteq\mathbb{R}^{14}$, consisting
of all vectors such that $x_0=x_2=x_5=y_0=y_2=y_5=0$.
Note that $||\mb{x}||\le C \eps$ (here and below, we use $C$ and $c$
to denote constants, whose exact value is irrelevant to the argument and which may be different from line to line,
but have no implicit dependence on any variable).
We will assume that $A'/2\Delta'\le A/2\Delta$,
and show that we necessarily then have that $M'=M$.

Consider the altitude dropped from each vertex $\mb{m}_i'$ of $M'$ to the opposite edge $\mb m_{i+3}'\mb m_{i+4}'$
and label the point of intersection $\mb{k}_i'$.
The chord $\mb{m}_i'\mb{k}_i'$ is an affine diameter of $M'$. Consider also for each $i$,
the two chords parallel to $\mb{m}_i'\mb{k}_i'$ but of half its length,
and let the parallelogram formed by them be of area $\Delta_i$. Let $\phi_i = \tfrac{A'/2\Delta_i}{A/2\Delta}-1$.
By our assumption, $\phi_i\le 0$ for all $i$.

Consider $\phi_i$ as a function of $\mb{x}$. This function depends analytically
on $\mb{x}$ in a neighborhood of the origin.
Within this neighborhood, we may bound $\phi_i(\mathbf{x})$ using its Taylor series
$$\phi_i(\mathbf{x}) \ge \langle \mb{f}_i,\mb{x}\rangle + \tfrac12 \langle\mb{x}, F_i\mb{x}\rangle -C ||\mb{x}||^3\text,$$
where the explicit values of $\mb{f}_i$ and $F_i$ are given in Tables \ref{tabfgen} and \ref{tabFF1}.

We note that there exist coefficients $c_i>0$ such that $\sum_{i=0}^6 c_i\mb{f}_i=0$, namely $c_i=1$ for all $i$.
It follows from the fundamental theorem of linear algebra that $\langle \mb{f}_i,\mb{x}\rangle\le 0$
for all $i$ if and only if
\begin{equation}\langle \mb{f}_i,\mathbf{x}\rangle = 0\text{ for all }i\text.\label{ns}\end{equation}
The intersection of the space of solutions to \eqref{ns} with $W$
is the two-dimensional space spanned by the two vectors given in Table \ref{tabfgen}.
We denote the orthogonal projection to this space as $P$. Note that (by a compactness 
argument) $\langle \mb{f}_i,\mb{x}\rangle \ge c ||(1-P)\mb{x}||$ for at least one $i$, and so
it follows from the assumption that $\phi_i(\mb{x})\le0$ for all $i$
and the fact that $\phi_i\ge\langle \mb{f}_i,\mb{x}_i\rangle -C ||\mb{x}||^2$
that $||(1-P)\mb{x}|| \le C ||\mb{x}||^2$. Therefore we also have that
$$\phi_i(\mb{x}) \ge \langle\mb{f}_i,(1-P)\mb{x}\rangle + \tfrac12\langle\mb{x},P F_i P \mb{x}\rangle-C||\mb{x}||^3\text.$$
By direct calculation, we observe that $P F_i P$ is positive definite (when restricted to the image of $P$) for all $i$,
and so $\tfrac12 \langle\mb{x},P F_i P \mb{x}\rangle\ge c||P\mb{x}||^2$. Therefore,
$0\ge\phi_i\ge c||(1-P)\mb{x}||+c'||P\mb{x}||^2-C||\mb{x}||^3$ for at least one $i$, and if $\eps$
is small enough then $\mb{x}=0$ necessarily.
\end{proof}

\begin{thm}\label{nonhthm}There exists $\eps>0$ such that if $(1-\eps)M\subseteq M'\subseteq(1+\eps)M$
then $A'/2\Delta'\ge A/2\Delta$, with equality only when $M'$ is affinely equivalent to $M$.\end{thm}

\begin{proof}
We now allow $M'$ to be an arbitrary domain, not necessarily a heptagon.
Consider the length of $M$ as a function of direction.
This function has seven local minima, corresponding to the chords from each vertex $\mb{m}_i$ to the midpoint
of the opposite edge $\mb{k}_i$. The corresponding function for $M'$ must also, when $\eps$ is sufficiently
small, have at least seven local minima realized by chords $\mb{m}_i'\mb{k}_i'$, where
$||\mb{m}_i'-\mb{m}_i||,||\mb{k}_i'-\mb{k}_i||< C\eps$ for all $i$. As in the previous proof,
let us denote $\mb{m}_i' = R^i (1+x_i,v y_i)$. Additionally, let
$\mb{k}_i''$ be the nearest point on the chord $\mb{m}_{i+3}'\mb{m}_{i+4}'$ to the point $\mb{m}_i'$
and let $\mb{k}_i' = \mb{k}_i'' + R^i (x_i',y_i')$ (see Figure \ref{lenfig}).

\begin{figure}
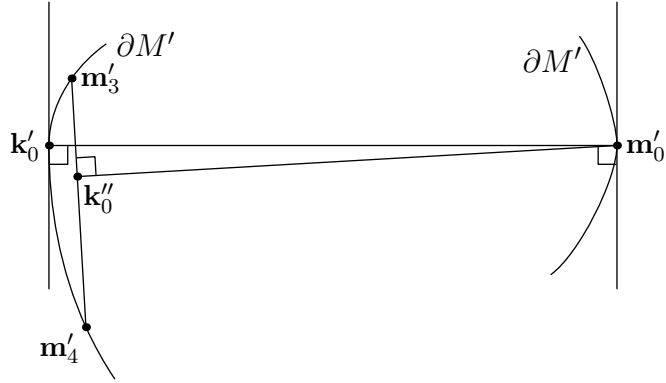
\begin{center}
\begin{asy}
import cseblack;
import olympiad;
import geometry;
usepackage("amssymb");
size(250);

pair M0=D("\mathbf{m}_0'",(1,0.2),E);
pair M3=D("\mathbf{m}_3'",rotate(3*360./7.)*(1,0),E);
pair M4=D("\mathbf{m}_4'",(rotate(4*360./7.)*(1,0))+(+0.05,0),SW);
pair K0P=D("\mathbf{k}_0'",M0+(-1.98,0),W);
path PM=M3+(0.12,0.12)..M3..K0P{down}..M4..M4+(0.1,-0.18);

MC("\partial M'",D(M0+(-0.23,-0.45)..tension 1.5 ..M0{up}..tension 1.5 ..M0+(-0.13,0.38)),0.9,W);
MC("\partial M'",D(PM),0.,E);
path CRD=D(M3--M4);
path ALT=M0--(M0+5*(rotate(90)*(M3-M4)));
pair K0PP=D("\mathbf{k}_0''",IP(ALT,CRD),SE);
D(M0--K0PP);
markrightangle(M0,K0PP,M3,size=2.5mm);
D(M0--K0P);
D(M0+(0,0.5)--M0-(0,0.5));
D(K0P+(0,0.5)--K0P-(0,0.5));
markrightangle(M0,K0P,K0P-(0,0.5),size=2.5mm);
markrightangle(M0-(0,0.5),M0,K0P,size=2.5mm);
\end{asy}
\caption{\label{lenfig}For a given domain $M'$ in the proof of Theorem \ref{nonhthm}, we identify
directions for which the length of $M'$ is a local minimum. For
example, in the illustration $\mb{m}_0'\mb{k}_0'$ is an affine diameter associated with
one of these directions. Other such affine diameters
originate at $\mb{m}_3'$ and $\mb{m}_4'$. To build more directly on the result for non-regular
heptagons of Theorem \ref{heptthm}, we give the coordinates of $(x_0',y_0')$ of $\mb{k}_0'$
in reference to $\mb{k}_0''$, the point closest to $\mb{m}_0'$ on the chord $\mb{m}_3'\mb{m}_4'$
(see text).}
\end{center}\end{figure}

For each chord $\mb{m}_i'\mb{m}_{i+1}'$ consider the arc of the boundary between
$\mb{m}_i'$ and $\mb{m}_{i+1}'$ as the graph of a
function $h_i(t)$, where $2vh_i(t)$ is the height of the boundary above the chord at the point
$(1-t)\mb{m}_i'+t\mb{m}_{i+1}'$ on the chord (see Figure \ref{hgtfig}).
Denote the corresponding boundary point $\mb{p}_i(t)$. The domain $M'$ is fully specified
by the points $\mb{m}_i'$ and $\mb{k}_i'$ and the functions $h_i(t)$, $i=0,\ldots,6$. However, we 
intend to use only the points $\mb{m}_i'$ and $\mb{k}_i'$ and the values $h_i(a)$, $h_i(b)$, $h_i(1-b)$,
and $h_i(1-a)$. Note that given the value of, say, $h_i(t_0)$, we can bound near-by values from convexity:
\begin{equation}\label{hbound}
\min\left(\frac{t}{t_0},\frac{1-t}{1-t_0}\right)\le \frac{h_i(t)}{h_i(t_0)}\le \max \left(\frac{t}{t_0},\frac{1-t}{1-t_0}\right)\text.
\end{equation}

\begin{figure}
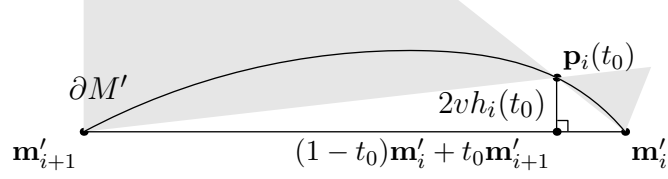
\begin{center}
\begin{asy}
import cseblack;
import olympiad;
import geometry;
usepackage("amssymb");
size(250);

real U=0.9009688679024191;
real A=1.75-2*U*U;

pair M0=D("\mathbf{m}_i'",(1,0),SE);
pair M1=D("\mathbf{m}_{i+1}'",(-1,0),SW);
pair MT=D("(1-t_0)\mathbf{m}_i'+t_0\mathbf{m}_{i+1}'",(1-A)*M0+A*M1,SW);
pair PT=D("\mathbf{p}_i(t_0)",MT+(0,0.2),NE);
pair PTP=MT+(-0.5,0.3);
real A1 = 0.2;
real A0 = 1.5;
pair S0=(1+A0)*PT-A0*M0;
pair S1=(1+A1)*PT-A1*M1;
pair S2=(M1.x,S0.y);

fill(S0--S2--M1--PT--cycle,lightgray);
fill(S1--M0--PT--cycle,lightgray);

MC("\partial M'",D(M0..PT..PTP..M1),0.9,NW);
D(M0--M1);
MC("2vh_i(t_0)",D(MT--PT),0.5,W);
markrightangle(M0,MT,PT,size=1.5mm);
\end{asy}
\caption{\label{hgtfig}The arc of the boundary of $M'$ between the points $\mathbf{m}_i'$ and
$\mathbf{m}_{i+1}'$ is given by the graph of the function $h_i(t)$. The highlighted
gray area marks the region where the boundary must lie according to \eqref{hbound}.}
\end{center}\end{figure}

Consider now the two chords parallel to $\mb{m}_i'\mb{k}_i'$ and half of its length. It is impossible
to determine the distance between them based on only the values we have decided to use. However,
we can bound it from above by replacing the actual boundary of $M'$ with the graph of the upper
bound given by \eqref{hbound}. Specifically, we replace the boundary above $\mb{m}_{i+1}\mb{m}_{i+2}$,
$\mb{m}_{i+2}\mb{m}_{i+3}$, $\mb{m}_{i+4}\mb{m}_{i+5}$, and $\mb{m}_{i+5}\mb{m}_{i+6}$, respectively with
the upper bound given by $t_0=a,b,(1-b),$ and $(1-a)$. We then find the chords of the replacement boundary arcs
that are parallel to $\mb{m}_i'\mb{k}_i'$ and half of its length, and call the area of the resulting
parallelogram $\Delta_i$. Note that $\Delta_i\ge\Delta'$, since $\Delta_i$ is no smaller than the area
of an actual half-length parallelogram inscribed in $M'$, which is in turn no smaller than
the smallest such area. Let $A''$ be the area of the polygon
$\mb{m}_0'\mb{p}_0(a)\mb{p}_0(b)\mb{k}_4'\mb{p}_0(1-b)\mb{p}_0(1-a)\mb{m}_1'\ldots\mb{p}_6(1-b)\mb{p}_6(1-a)$, so we
have $A''\le A'$. We will assume that $A'/2\Delta'\le A/2\Delta$, and show that this necessarily
implies that $M'$ is affinely equivalent to $M$. Since $A''/2\Delta_i\le A'/2\Delta'$, then
$\phi_i=\tfrac{A''/2\Delta_i}{A/2\Delta}-1\le 0$.

Let us consider $\phi_i$ as a function of $\mb{x}=(x_0,y_0,x_1,y_1,\ldots,y_6)\in\mathbb{R}^{14}$, and
$\mb{x}'=(x_0',y_0',x_1',y_1',\ldots,y_6',\allowbreak h_0(a),h_0(b),h_0(1-b),
\allowbreak h_0(1-a),h_1(a),\allowbreak \ldots,h_6(1-a))\in\mathbb{R}^{42}$.
In contrast to the last proof, here $\phi_i$ is not analytic in any neighborhood of
the origin in $\mathbb{R}^{14}\times\mathbb{R}^{42}$. However,
it does, everywhere in such a neighborhood, take the value of one of 16 analytic functions
(let us call them $\phi_{ij}(\mb{x},\mb{x}')$, $j=1,\ldots,16$),
based on whether $t>t_0$ or not at the point of contact of the paralellogram with each of
the four replacement boundary arcs. When $\mb{x}'=0$ all sixteen functions
agree. Also, the first derivatives of $\phi_{ij}$ with respect to any component
taken at the origin are independent of $j$. Therefore, we have that 
$$\phi_i(\mb{x},\mb{x}') \ge \phi_i(\mb{x},0) + \langle \mb{f}_i',\mb{x}'\rangle - C||\mb x'||(||\mb x||+||\mb x'||)\text.$$
Note that $\phi_i(\mb{x},0)\ge \tfrac{A_0/2\Delta_0}{A/2\Delta}-1$, where $A_0/2\Delta_0$ is the
double-lattice packing density of the heptagon $M_0=\mb{m}_0'\mb{m}_1'\ldots\mb{m}_6'$.
From Theorem \ref{heptthm} it then follows that $\phi_i(\mb{x},0)\ge0$.
For explicit values of $\mb{f}_i'$, see Table \ref{tabfgp}.

We now consider additional functions $\psi_i(\mb{x},\mb{x}')$, $i=1,\ldots,42$ given by
(in each of the definitions that follow $i=1\ldots,7$)
$$\begin{aligned}
\psi_i &= \langle\mb{k}_{i+4}'-\mb{m}_{i+4}',\mb{k}_{i+4}'-\mb{p}_{i}(b)\rangle \\ 
\psi_{i+7} &= \langle\mb{k}_{i+4}'-\mb{m}_{i+4}',\mb{k}_{i+4}'-\mb{p}_{i}(1-b)\rangle  \\ 
\psi_{i+14} &= \alpha(\mb{p}_i(a),\mb{p}_i(b),\mb{k}_{i+4}')  \\ 
\psi_{i+21} &= \alpha(\mb{k}_{i+4}',\mb{p}_i(1-b),\mb{p}_i(1-a))  \\ 
\psi_{i+28} &= \alpha(\mb{m}_i',\mb{p}_i(a),\mb{p}_i(b))  \\ 
\psi_{i+35} &= \alpha(\mb{p}_i(1-b),\mb{p}_i(1-a),\mb{m}_{i+1}') \text, 
\end{aligned}$$
where
$$\alpha(\mb{p},\mb{p}',\mb{p}'') = \mb{p}\wedge\mb{p}' + \mb{p}'\wedge\mb{p}'' + \mb{p}''\wedge\mb{p}$$
is the oriented area of the triangle $\mb{p}\mb{p}'\mb{p}''$.
From the fact that $\mb{m}_i'\mb{k}_{i}'$ is a locally shortest length, we have that a line through $\mb{k}_i'$
perpendicular to this length is tangent to $M'$, and therefore $\psi_i\ge0$ for $i=1,\ldots,14$.
That $\psi_i\ge0$ for $i=15,\ldots,42$ simply follows from convexity. These functions are all
analytic in a neighborhood of the origin, and therefore we have that 
$$\psi_i(\mb{x},\mb{x}') \le \psi_i(\mb{x},0) + \langle \mb{g}_i',\mb{x}'\rangle + C||\mb x'||(||\mb x||+||\mb x'||)\text.$$
Moreover, note that $\psi_i(\mb{x},0)=0$ for all $i=1,\ldots,42$.
For explicit values of $\mb{g}_i'$, see Table \ref{tabfgp}.

There exist coefficients $c_i>0$, $i=0,\ldots,6$, and $d_i>0$, $i=1,\ldots,42$, such
that $\sum_{i=0}^6 c_i\mb{f}_i'-\sum_{i=1}^{42} d_i\mb{g}_i'=0$.
It then follows from the fundamental theorem of linear algebra that if
$$\begin{aligned}
\langle \mb{f}_i',\mb{x}'\rangle&\le 0\text{ for }i=0,\ldots,6\text{ and }\\
\langle \mb{g}_i',\mb{x}'\rangle&\ge 0\text{ for }i=1,\ldots,42\text,
\end{aligned}$$
then we have equality for all of the above. The solution space turns out to be trivial.
From compactness there must be a constant $C$ such that at least one of the following equations holds for at least one $i$
$$\begin{aligned}
\langle \mb{f}_i',\mb{x}'\rangle&\ge C||\mb{x}'||\text{ or}\\
\langle \mb{g}_i',\mb{x}'\rangle&\le -C||\mb{x}'||\text.
\end{aligned}$$
Therefore, it follows from the fact that $\phi_i\le0$ and $\psi_i\ge0$ for all $i$,
that there exists $\eps$ such that if $||\mb{x}||,||\mb{x}'||<\eps$ then $\mb{x}'=0$.
If $\mb{x}'=0$, then from convexity $h_i(t)=0$ for all $i=0,\ldots,6$ and $0\le t\le1$,
and $M'$ is a heptagon. From Theorem \ref{heptthm} $M'$ is affinely equivalent to $M$.
\end{proof}

\begin{cnj}The regular heptagon is an absolute minimum of $\delta_{L^*}$ in
two dimensions.\end{cnj}

If, as might very well be the case, $\delta(M)=\delta_{L^*}(M)$, then
the conjecture would also imply that $M$ is a minimum of $\delta$.

\section{The 3-ball}

Let $K$ and $K'$ be convex bodies and let $\lambda,\lambda'\ge0$ not both equal to $0$, then
the set $\lambda K+\lambda' K' = \{\lambda\mathbf{x}+\lambda'\mathbf{x}': \mathbf{x}\in K, \mathbf{x}\in K'\}$
is also a convex body. This operation is known as the Minkowski sum.
A convex body $K\subseteq\mathbb{R}^n$ can be specified by its support height function $h_K:S^{n-1}\to\mathbb{R}$,
given by $h_K(\mb{x}) = \max_{\mb{y}\in K} \langle\mb{x},\mb{y}\rangle$. Minkowski addition corresponds
to addition of the support height functions: $h_{\lambda K+\lambda' K'}(\mb{x}) = \lambda h_K(\mb{x}) + \lambda' h_{K'}(\mb{x})$.
The mean width of a body $K$ is the average length of its projection onto
a randomly chosen axis. It is given by 
$$w = \frac{2\int_{S^{n-1}} h_K d\sigma}{\sigma(S^{n-1})}\text,$$
where $\sigma$ is the Lebesgue measure on $S^{n-1}$.
Out of all linear images $TK$ of a body $K$, there
is a unique one up to rotation which minimizes $w$
while preserving the volume \cite{minwidth}. This is known as the minimal-mean-width
position of $K$, and a body is in its minimal-mean-width position
if and only if
$$\int_{S^{n-1}} h_K(\mathbf{x}) \langle\cdot,\mathbf{x}\rangle^2 d\sigma(\mathbf{x}) = (w/2n)\sigma(S^{n-1})||\cdot||^2\text.$$
Steiner's formula gives the volume of the body $K_\lambda = (1-\lambda)B+\lambda K$,
interpolating between the unit ball $B$ ($\lambda=0$) and the body $K$ ($\lambda=1$).
In three dimensions, Steiner's formula can be written as
\begin{equation}\label{steineq}
\vll(K) = \tfrac{4\pi}{3} (1-\lambda)^3 + 2\pi w \lambda + S(K) \lambda^2(1-\lambda) + \lambda^3 \vll(K)\text,
\end{equation}
where $S(K)$ is the surface area of $K$ \cite{schneider}.

In this section we prove the following result about the unit ball
and the double-lattice packing density of nearly spherical bodies:

\begin{thm}\label{b3thm}Let $K$ be a three-dimensional body in minimal-mean-width position.
If $K$ is not a ball,
then there exist numbers $\lambda_0(K)>0$ and $\beta(K)>0$ such that
$$\delta_{L^*}\left((1-\lambda)B+\lambda K\right) -\delta_{L^*}(B) > \beta(K) \lambda\text,$$ 
for all $0<\lambda<\lambda_0(K)$.\end{thm}

The double-lattice packing density of $B$ is $\pi/\sqrt{18}$. It is realized,
for example, by its optimal lattice packing, the face-centered cubic lattice (f.c.c.),
which can be described degenerately as a double lattice. It is also realized
by the hexagonally closed packed structure (h.c.p.), which is not a Bravais
lattice. We fix a realization of the \hcp structure in which the unit ball
centered at the origin shares the following twelve boundary points with neighboring balls:
$\mb{x}_1=(1,0,0)$, $\mb{x}_2=(1/2,\sqrt{3}/2,0)$,
$\mb{x}_3=\mb{x}_2-\mb{x}_1$, $\mb{x}_4=-\mb{x}_1$, $\mb{x}_5=-\mb{x}_2$, $\mb{x}_6=\mb{x}_1-\mb{x}_2$,
$\mb{x}_7=(1/2,1/\sqrt{12},\sqrt{2/3})$, $\mb{x}_8=\mb{x}_7-\mb{x}_1$, $\mb{x}_{9}=\mb{x}_7-\mb{x}_2$
$\mb{x}_{10}=(1/2,1/\sqrt{12},-\sqrt{2/3})$, $\mb{x}_{11}=\mb{x}_{10}-\mb{x}_1$, and $\mb{x}_{12}=\mb{x}_{10}-\mb{x}_2$.
The double lattice $\Xi$, of mean volume $d(\Xi)=4\sqrt{2}$, is generated by translations by
$2\mb{x}_1$ and $2\mb{x}_2$ and by
inversions about $\mb{x}_7$ and $\mb{x}_{10}$.
Let $P$ be the polyhedron
$\{\mb{x}\in\mathbb{R}^3:\langle\mb{x},\mb{x}_i\rangle\le1\text{ for all }i=1,\ldots,12\}$,
then $P$ is the Voronoi cell of the \hcp structure.

The double lattice $\Xi$
is admissible for $P$ too, producing a packing of density 1, namely a tiling. Specifically,
this is a face-to-face tiling in the strong sense that every two cells share a face or do not
touch at all. We show now that if $K$ is a nearly spherical convex body,
we can bound its double-lattice packing density using the values $h_K(\mb{x}_i)$, $i=1,\ldots,12$.

\begin{lem}\label{lemcns} Let $K$ be a convex body satisfying $(1-\eps)B\subseteq K\subseteq(1+\eps) B$.
For sufficiently small $\eps$, a double-lattice $\Xi'$ exists, admissible for $K$, such that
$d(\Xi')\le d(\Xi) \eta(K)^3$, where $\eta(K)=\tfrac{1}{12}\sum_{i=1}^{12} h_K(\mb{x}_i)$.\end{lem}

\begin{proof}
Without loss of generality, let us assume that $\sum_{i=1}^{12} h_K(\mb{x}_i)=12$.
Let us label $h_i=h_K(\mb{x}_i)$ and
consider be the polyhedron $P'=\{\mb{x}\in\mathbb{R}^3:\langle\mb{x},\mb{x}_i\rangle\le h_i\text{ for all }i=1,\ldots,12\}$.
The projection of $P'$ onto the $xy$-plane is a hexagon. Let $\mb{a}_1$ and $\mb{a}_2$
be the vectors in the $xy$-plane generating the densest lattice packing of this hexagon.
In particular, for $\eps$ small enough, there is a unique choice such that $||\mb{a}_1-2\mb{x}_2||,||\mb{a}_2-2\mb{x}_2||<C\eps$.
Now, let $\mb{x}_7'$, $\mb{x}_8'$, and $\mb{x}_9'$, be the unique points satisfying
$\langle\mb{x}_i',\mb{x}_i\rangle=h_i$ for $i=7,8,9$,
$\mb{x}_8' = \mb{x}_7' - \tfrac{1}{2}\mb{a}_1$, and
$\mb{x}_9' = \mb{x}_7' - \tfrac{1}{2}\mb{a}_2$.
Similarly, let $\mb{x}_{10}'$, $\mb{x}_{11}'$, and $\mb{x}_{12}'$, be the unique points satisfying
$\langle\mb{x}_i',\mb{x}_i\rangle=h_i$ for $i={10},{11},{12}$,
$\mb{x}_{11}' = \mb{x}_{10}' - \tfrac{1}{2}\mb{a}_1$, and
$\mb{x}_{12}' = \mb{x}_{10}' - \tfrac{1}{2}\mb{a}_2$.
Now let $\Xi'$ be the double lattice generated by translations by $\mb{a}_1$
and $\mb{a}_2$ and by inversions about $\mb{x}_7'$ and $\mb{x}_{10}'$.
We note that for each face of $P'$ there is a neighbor $\xi'(P')$, $\xi'\in\Xi'$,
such that $P'$ and $\xi'(P')$ touch along this face. For small enough
$\eps$, this is enough to conclude that $\Xi'$ is admissible for $P'$,
since in the packing $\Xi(P)$ there are only face-to-face contacts. {\it A
fortiori}, $\Xi'$ is also admissible for $K$.

As $\mb{a}_1$, $\mb{a}_2$, $\mb{x}_7'$ and $\mb{x}_{10}'$ may be
determined explicitly as a function of $h_i$, $i=1,\ldots,12$,
we calculate the mean volume of $\Xi'$ to be
$$\begin{aligned}d(\Xi') =& 4\sqrt{2} - \tfrac{\sqrt{2}}{9} (\eta_1+\eta_2+\eta_3)^2 - \tfrac{2\sqrt{2}}{3} (\eta_1^2+\eta_2^2+\eta_3^2)\\
&+\tfrac{\sqrt{2}}{9}(\eta_1+\eta_2+\eta_3)\left(2(\eta_1^2+\eta_2^2+\eta_3^2)-(\eta_1+\eta_2+\eta_3)^2\right)\text,\end{aligned}$$
where $\eta_1=h_1+h_4-2$, $\eta_2=h_2+h_5-2$, and $\eta_3=h_3+h_6-2$. Note that the quadratic term
is negative unless $\eta_1=\eta_2=\eta_3=0$, in which case the quadratic and cubic term both vanish.
Therefore, when $\eps$ is small enough $d(\Xi')\le 4\sqrt{2}$.
\end{proof}

\begin{lem}\label{lemcl}Let 
$$c_l=P_l(1)+4P_l(\tfrac{1}{2})+2P_l(0)+P_l(-\tfrac{1}{3})+2P_l(-\tfrac{1}{2})+2P_l(-\tfrac{5}{6})\text,$$
where $P_l(t)$ is the Legendre polynomial of degree $l$. Then $c_l=0$ if and only if $l=1$ or $l=2$. \end{lem}
\begin{proof}
We introduce the following rescaled Legendre polynomials:
$Q_l(t)\allowbreak =\allowbreak 6^l l! P_l(t)$. From their recurrence relation---given by
$Q_{l+1}(t) = (2l+1) (6t) Q_l(t) - 36 l^2 Q_{l-1}(t)$---and the base cases---$Q_0(t)=1$ and $Q_1(t)=6t$---
it is clear that the values of $Q_l(t)$ at $t=k/6$ for $k=-6,\ldots,6$ are integers.
If $Q_l(k/6)\equiv Q_{l+1}(k/6)\equiv0\pmod{8}$ for some $k$ and $l$ then for all $l'\ge l$,
$Q_{l'}(k/6)\equiv0\pmod{8}$.
This is the case for $k=0,2,6$ and $l=3$, as can be easily checked.

For $k=3$ and $k=5$ it is easy to show by induction that
the residue of $Q_l(k/6)$ modulo $8$ depends only on $k$ and the residue of $l$ modulo $4$ and
takes the following values:
$$\begin{aligned}
Q_l(1/2) &\equiv& 1,3,7, 1 \pmod{8}\\
Q_l(5/6) &\equiv& 1,5,7, 7 \pmod{8}\\
\text{ resp. for } l&\equiv& 0,1,2, 3 \pmod{4}\text.
\end{aligned}$$
Therefore, when $l\ge3$ is odd, then
$6^l l! c_l = Q_l(1)-2Q_l(\tfrac{5}{6})+2Q_l(\tfrac{1}{2})-Q_l(\tfrac{1}{3})\equiv 4 \pmod{8}$,
and therefore cannot vanish. When $l\ge3$ is even,
then $6^l l! c_l = Q_l(1)+2Q_l(\tfrac{5}{6})+6Q_l(\tfrac{1}{2})+Q_l(\tfrac{1}{3})+2Q_l(0)\equiv 8 \pmod{16}$,
and again cannot vanish. This leaves
only the cases $c_0=12$, $c_1=0$, and $c_2=0$ to be calculated manually.
\end{proof}

\begin{lem}Let $K$ be a three-dimensional body in minimal-mean-width position. If $K$ is not a ball
then there is a body $K'$, isometric to $K$, such that
$\eta(K')<\tfrac{1}{2}w$, where $w$ is the mean width of $K$.\end{lem}
\begin{proof}
Note that if $K'=R(K)$ is a rotation of $K$ about the origin, then $h_{K'}(\mb{x})=h_K(R^T\mb{x})$.
Let us pick a point $\mb{y}\in S^2$, and let $R_0$ be some rotation such that $R_0(\mb{y})=\mb{x}_7$.
Let $R_\theta$ be the rotation obtained by composing $R_0$ with a rotation by $\theta$ about
the axis through $\mb{x}_7$, so that $R_\theta(\mb{y})=\mb{x}_7$ for all $0\le\theta\le2\pi$.
Now let $g(\mb{y}) = (1/2\pi)\int_0^{2\pi}\eta(R_\theta(K))d\theta$, and repeat this
definition for all $\mb{y}\in S^2$.

The function $g(\mb{y})$ is given by integrating $h_K(\mb{x})$ over
a measure $\mu_\mb{y}(\mb{x})$. The measures $\mu_\mb{y}(\mb{x})$ are
each invariant under rotations about the axis through $\mb{y}$, and are
related to each other by rotations. Such an operation $h_K(\mb{x})\mapsto g(\mb{y})$
is known as a convolution by the zonal measure $\mu_\mb{p}(\mb{x})$,
where $\mb{p}$ is some arbitrary pole.
(see Ref. \cite{convolutions} for results about convolutions with
zonal measures). The measure $\mu_\mb{p}$ is given by
$$12 \mu_\mb{p}(\lbrace\mb{x}:\langle\mb p,\mb x\rangle \in (t_1,t_2)\rbrace) =
\left|\lbrace i\in\lbrace1,2,\ldots,12\rbrace : \langle\mb{x}_7,\mb{x}_i\rangle\in(t_1,t_2)\rbrace\right|\text.$$

We can expand $\mu_\mb{p}(\mb x)$ into spherical harmonics to
obtain $$\mu_\mb{p}(\mb{x}) = \tfrac{1}{12}\sum_{l=0}^{\infty} c_l P_l(\langle \mb{x},\mb{p}\rangle)\text,$$
where $c_l$ are the coefficients of Lemma \ref{lemcl}.
It follows that if $h_K(\mb{x}) = \sum_{l=0}^{\infty} h_l(\mb{x})$ is
the expansion of $h_K$ into spherical harmonics, then
$g(\mb{x}) = \tfrac{1}{12}\sum_{l=0}^{\infty} c_l h_l(\mb{x})$ \cite{convolutions}.
The $l=0$ term of $g(\mb{x})$, giving its average value, is equal
to that of $h_K(\mb{x})$, namely $w/2$. Because $K$ is in minimal-mean-width
position, $h_2=0$. Therefore, by Lemma \ref{lemcl}, $g(\mb{x})$
is constant if and only if the spherical harmonics expansion of $h_K$
terminates at $l=1$, which in turn is equivalent to $K$ being a ball.
Since we assume $K$ is not a ball, then $g(\mb{x})$ is not constant
and must achieve a value below its average. Since this value corresponds
in turn to an average of values of $\eta(R (K))$ over a set of
rotations $R$, it must be no smaller than the minimum value
among these rotations. Therefore, there is a rotation $R$
such that $\eta(R(K))<w/2$.
\end{proof}

We now prove Theorem \ref{b3thm}.

\begin{proof}Without loss of generality, we may assume that $\vll(K)=\vll(B)$
and that $K$ is rotated such that $\eta(K)<w/2$. Let $K_\lambda = (1-\lambda)B+\lambda K$, then
$\eta(K_\lambda) = 1+(\eta(K)-1)\lambda$. The isoperimetric inequality, $S(K)>S(B)$,
and Steiner's formula \eqref{steineq} give
$$\frac{\vll(K_\lambda)}{\vll(B)}\ge 1+3(\tfrac{w}{2}-1)\lambda(1-\lambda)^2\text.$$
The claim of the theorem now follows immediately from Lemma \ref{lemcns}.\end{proof}

\begin{cnj} 
The ball is a local minimum of $\delta_{L^*}$ in two dimensions.\end{cnj}

It does not seem that the ball is a global minimum
of $\delta_{L^*}$. For example, the densest-known double-lattice packing
of the tetrahedron $T$ has a density of only $\tfrac{1}{369}(139+40\sqrt{10})=0.71948\ldots$,
so probably $\delta_{L^*}(T)<\delta_{L^*}(B)$ \cite{kallusdcg}. Still, if the conjecture
holds, then the ball would also be a local minimum of $\delta$, verifying
a local version of Ulam's conjecture.

\section{Discussion}

We conclude with a summary of known results and open problems.
Recall from the introduction that problems 1--3 ask for
bodies that minimize the functions $\delta$, $\delta_{L}$,
$\delta_{L^*}$, $\delta_{T}$, or $\delta_{T^*}$ either
among all convex bodies or among only \cs bodies.
There are only two case that are solved: the minimum
of $\delta_L$ and $\delta_T$ in two dimensions among
all convex bodies is
$2/3$, as realized by triangles alone \cite{Fary,Courant}.
A. Bezdek and W. Kuperberg comment that determining
the minima in the unsolved cases ``seems to be a very challenging
problem, perhaps too difficult to expect to be solved in
foreseeable future'' \cite{BezKup}. Determining local minima
seems to be a more approachable problem, and so far the following
local minima have been identified:

\begin{itemize}
\item Reinhardt's smoothed octagon is a local minimum of $\delta_{L}$
and of $\delta_{T}$ among \cs bodies in two dimensions \cite{Nazarov}.

\item The ball is a local minimum of $\delta_{L}$ and of $\delta_{T}$
among \cs bodies in three dimensions \cite{kalluspack}.

\item The regular heptagon is a local minimum of $\delta_{L^*}$
among all convex bodies in two dimensions (Section 2).
\end{itemize}

Note that the present work is the only case in the list above
of a local minimum among all convex bodies. Reinhardt's smoothed
octagon possesses the property that its lattice packing density
is achieved simultaneously by a one-parameter family of lattices (see Figure
\ref{packfig}).
In fact this property, common to all so-called irreducible
domains (domains all of whose proper subdomains have
admissible lattices of lower mean area), has long been a central
organizing idea in the study
of Reinhardt's conjecture \cite{mahler98}. Therefore, it might be surprising
to some that the heptagon, despite being irreducible
with respect to double lattices, does not have a one parameter
family of optimal admissible double lattices.

We end with three open problems:

\begin{itemize}
\item The conjecture that the ball is a global minimum of $\delta$ among
convex bodies in three dimensions has
been attributed to Ulam \cite{Gardner95}. A weaker claim,
that the ball is a local minimum of $\delta$, is open. It is also possible
that the ball is a local minimum of $\delta_{T^*}$ or of $\delta_{L^*}$,
as we conjecture here (Section 3). Either
of these possibilities necessarily imply that the ball is a local minimum of $\delta$,
but they do not necessarily follow from Ulam's conjecture.

\item The regular heptagon is conjectured to be a local minimum
of $\delta$ among convex bodies in two dimensions. This would
follow immediately if the packing density of the regular heptagon
is shown to be equal to its double-lattice packing density.

\item In four dimensions, it is known that the ball is not a minimum of $\delta_L$
among \cs bodies \cite{kalluspack}. It would be interesting to identify a body which
is such a minimum.
\end{itemize}

{\bf Acknowledgments.} I would like to thank Wlodzimierz Kuperberg for his helpful comments.

\bibliographystyle{amsplain}
\bibliography{worst.bib}

\begin{table}[b]
\resizebox{\textwidth}{!}{
\begin{tabular}{c|ccc}
& $\mb{f}_0$ & $\mb{u}_1$ &$\mb{u}_2$ \\\hline
$\langle \cdot, \mb{e}_{1}\rangle$ & $-\tfrac{2}{679}( 419 - 452 u + 40 u^2) $ & $ 0 $ & $ 0 $ \\
$\langle \cdot, \mb{e}_{2}\rangle$ & $0 $ & $ 0 $ & $ 0 $ \\
$\langle \cdot, \mb{e}_{3}\rangle$ & $\tfrac{1}{1358}( 587 -  148 u - 732 u^2) $ & $ \tfrac{1}{1609} (-1171 + 1296 u + 3652 u^2) v $ & $ \tfrac{1}{3218}    (2631 - 194 u - 60 u^2) $ \\
$\langle \cdot, \mb{e}_{4}\rangle$ & $ -\tfrac{4}{679}( -81 - 8 u +  199 u^2) v $ & $ \tfrac{1}{3218}    (2631 - 194 u - 60 u^2) $ & $ -\tfrac{1}{11263} (-13669 + 13864 u + 19084 u^2)    v $ \\
$\langle \cdot, \mb{e}_{5}\rangle$ & $ \tfrac{3}{679}( -39 + 68 u +  6 u^2) $ & $ 0 $ & $ 0 $ \\
$\langle \cdot, \mb{e}_{6}\rangle$ & $ \tfrac{6}{679}( 76 - 135 u + 48 u^2) v $ & $ 0 $ & $ 0 $ \\
$\langle \cdot, \mb{e}_{7}\rangle$ & $ \tfrac{1}{14}( 5 -  12 u + 8 u^2) $ & $ \tfrac{1}{1609} (-1911 + 2104 u + 4300 u^2) v $ & $ \tfrac{1}{3218} (-1241 - 454 u + 1452 u^2) $ \\
$\langle \cdot, \mb{e}_{8}\rangle$ & $ \tfrac{2}{97}( -46 + 23 u +  22 u^2) v $ & $ \tfrac{1}{3218} (-295 + 670 u - 1916 u^2) $ & $ -\tfrac{1}{11263} (-10219 + 6792 u +     19020 u^2) v $ \\
$\langle \cdot, \mb{e}_{9}\rangle$ & $ \tfrac{1}{14}( 5 - 12 u + 8 u^2) $ & $ -\tfrac{2}{1609} (-930 + 885 u +     2032 u^2)    v $ & $ \tfrac{1}{1609} (561 - 456 u + 622 u^2) $ \\
$\langle \cdot, \mb{e}_{10}\rangle$ & $ -\tfrac{2}{97}( -46 +  23 u + 22 u^2) v $ & $ -\tfrac{2}{1609} (-140 - 282 u +     427 u^2) $ & $ \tfrac{2}{11263} (-5336 + 6583 u + 7908 u^2) v $ \\
$\langle \cdot, \mb{e}_{11}\rangle$ &$ \tfrac{3}{679}( -39 + 68 u +  6 u^2) $ & $ 0 $ & $ 0 $ \\
$\langle \cdot, \mb{e}_{12}\rangle$ &$ -\tfrac{6}{679}( 76 - 135 u +  48 u^2) v $ & $ 0 $ & $ 0 $ \\
$\langle \cdot, \mb{e}_{13}\rangle$ &$ \tfrac{1}{1358}( 587 - 148 u -  732 u^2) $ & $ 0 $ & $ 1 $ \\
$\langle \cdot, \mb{e}_{14}\rangle$ &$ \tfrac{4}{679}( -81 - 8 u + 199 u^2) v $ & $ 1 $ & $ 0 $

\end{tabular}}
\caption{\label{tabfgen}The left column gives the elements of $\mb{f}_0$ in the standard basis of $\mathbb{R}^{14}$. The elements of $\mb{f}_i$
are obtained by a cyclic permutation of the indices by $2i$. The other two columns give the elements of vectors such that
$a_1\mb{u}_1+a_2\mb{u}_2$ is the general solution satisfying the equations \eqref{ns} and $x_0=x_2=x_5=y_0=y_2=y_5=0$.}
\end{table}

\begin{table}[p]
\resizebox{\textwidth}{!}{
\begin{tabular}{c|cc}
&$\mb{e}_1$&$\mb{e}_2$\\\hline
$\langle \cdot, F_0 \mb{e}_{1}\rangle$ & $ (8/461041) (194143 - 526054 u + 360624 u^2) $ &	$ 0 $ \\
$\langle \cdot, F_0 \mb{e}_{2}\rangle$ & $ 0 $ &	$ 0 $ \\
$\langle \cdot, F_0 \mb{e}_{3}\rangle$ & $ -(2/461041)    (72529 - 63570 u + 30298 u^2) $ &	$ -(4/49)( -4 - u + 10 u^2 )v $ \\
$\langle \cdot, F_0 \mb{e}_{4}\rangle$ & $ (4/   461041) (-60408 - 31391 u + 115516 u^2) v $ &	$ 1/7 $ \\
$\langle \cdot, F_0 \mb{e}_{5}\rangle$ & $ (2/461041)    (19669 - 35394 u + 47802 u^2) $ &	$ 0 $ \\
$\langle \cdot, F_0 \mb{e}_{6}\rangle$ & $ -(4/461041)    (88155 - 230075 u + 174752 u^2) v $ &	$ 0 $ \\
$\langle \cdot, F_0 \mb{e}_{7}\rangle$ & $ (1/   679) (-569 + 2564 u - 2136 u^2) $ &	$ (4/679) (31 - 258 u + 40 u^2)    v $ \\
$\langle \cdot, F_0 \mb{e}_{8}\rangle$ & $ (4/65863)    (33298 - 78729 u + 50406 u^2) v $ &	$ -(4/679)    (129 - 113 u + 10 u^2) $ \\
$\langle \cdot, F_0 \mb{e}_{9}\rangle$ & $ (1/   679) (-569 + 2564 u - 2136 u^2) $ &	$ -(4/679) (31 - 258 u + 40 u^2)    v $ \\
$\langle \cdot, F_0 \mb{e}_{10}\rangle$ & $ -(4/65863) (33298 - 78729 u + 50406 u^2)    v $ &	$ -(4/679) (129 - 113 u + 10 u^2) $ \\
$\langle \cdot, F_0 \mb{e}_{11}\rangle$ & $ (2/461041)    (19669 - 35394 u + 47802 u^2) $ &	$ 0 $ \\
$\langle \cdot, F_0 \mb{e}_{12}\rangle$ & $ (4/461041)    (88155 - 230075 u + 174752 u^2) v $ &	$ 0 $ \\
$\langle \cdot, F_0 \mb{e}_{13}\rangle$ & $ -(2/461041)    (72529 - 63570 u + 30298 u^2) $ &	$ (4/49) (-4 - u + 10 u^2) v $ \\
$\langle \cdot, F_0 \mb{e}_{14}\rangle$ & $ -(4/461041) (-60408 - 31391 u + 115516 u^2) v $ &	$ 1/7 $ \\\hline
&$\mb{e}_{3}$&$\mb{e}_{4}$\\\hline
$\langle \cdot, F_0 \mb{e}_{1}\rangle$ & $ -(2/461041)    (72529 - 63570 u + 30298 u^2) $ &	$ (4/   461041) (-60408 - 31391 u + 115516 u^2) v $ \\
$\langle \cdot, F_0 \mb{e}_{2}\rangle$ & $ -(4/49) (-4 - u + 10 u^2) v $ &	$ 1/   7 $ \\
$\langle \cdot, F_0 \mb{e}_{3}\rangle$ & $ (1/   461041) (-20089 - 45197 u - 24338 u^2) $ &	$ -(1/461041)    (15943 + 262998 u + 380160 u^2) v $ \\
$\langle \cdot, F_0 \mb{e}_{4}\rangle$ & $ -(1/461041)    (15943 + 262998 u + 380160 u^2) v $ &	$ (1/   65863) (-23109 - 25645 u + 18926 u^2) $ \\
$\langle \cdot, F_0 \mb{e}_{5}\rangle$ & $ (1/461041)    (759 + 107533 u + 100614 u^2) $ &	$ (3/461041)    (8985 + 89258 u + 119032 u^2) v $ \\
$\langle \cdot, F_0 \mb{e}_{6}\rangle$ & $ (1/   461041) (-149167 - 245274 u + 634392 u^2) v $ &	$ (1/65863)    (20449 + 1367 u - 33410 u^2) $ \\
$\langle \cdot, F_0 \mb{e}_{7}\rangle$ & $ (1/9506) (-1347 + 2080 u + 13444 u^2) $ &	$ (2/   4753) (-439 + 1432 u + 3082 u^2) v $ \\
$\langle \cdot, F_0 \mb{e}_{8}\rangle$ & $ (2/65863) (-7673 + 14458 u + 17178 u^2)    v $ &	$ (1/131726) (66455 + 40832 u - 105836 u^2) $ \\
$\langle \cdot, F_0 \mb{e}_{9}\rangle$ & $ -(3/9506) (-587 + 2088 u + 3060 u^2) $ &	$ -(2/4753) (-583 +     1116 u + 3134 u^2) v $ \\
$\langle \cdot, F_0 \mb{e}_{10}\rangle$ & $ (2/65863)    (7479 + 13866 u + 26278 u^2) v $ &	$ (1/131726)    (65465 + 43752 u - 88164 u^2) $ \\
$\langle \cdot, F_0 \mb{e}_{11}\rangle$ & $ (1/   922082) (-166195 + 68402 u + 95304 u^2) $ &	$ (1/   461041) (-96429 - 46894 u + 182108 u^2) v $ \\
$\langle \cdot, F_0 \mb{e}_{12}\rangle$ & $ -(3/461041)    (41199 - 45674 u + 5436 u^2) v $ &	$ (3/   131726) (-14829 - 8154 u + 26800 u^2) $ \\
$\langle \cdot, F_0 \mb{e}_{13}\rangle$ & $ (1/   461041)    (28120 + 29493 u - 21622 u^2) $ &	$ -(1/461041) (-79117 - 84650 u +     70536 u^2) v $ \\
$\langle \cdot, F_0 \mb{e}_{14}\rangle$ & $ (1/   461041) (-79117 - 84650 u + 70536 u^2) v $ &	$ (1/   65863) (-21996 - 12961 u + 35782 u^2) $ \\\hline
&$\mb{e}_{5}$&$\mb{e}_{6}$\\\hline
$\langle \cdot, F_0 \mb{e}_{1}\rangle$ & $ (2/461041)    (19669 - 35394 u + 47802 u^2) $ &	$ -(4/461041)    (88155 - 230075 u + 174752 u^2) v $ \\
$\langle \cdot, F_0 \mb{e}_{2}\rangle$ & $ 0 $ &	$ 0 $ \\
$\langle \cdot, F_0 \mb{e}_{3}\rangle$ & $ (1/461041)    (759 + 107533 u + 100614 u^2) $ &	$ (1/   461041) (-149167 - 245274 u + 634392 u^2) v $ \\
$\langle \cdot, F_0 \mb{e}_{4}\rangle$ & $ (3/461041)    (8985 + 89258 u + 119032 u^2) v $ &	$ (1/65863)    (20449 + 1367 u - 33410 u^2) $ \\
$\langle \cdot, F_0 \mb{e}_{5}\rangle$ & $ (1/922082)    (213769 - 567974 u - 163748 u^2) $ &	$ -(3/461041)    (68953 - 192732 u + 128460 u^2) v $ \\
$\langle \cdot, F_0 \mb{e}_{6}\rangle$ & $ -(3/461041)    (68953 - 192732 u + 128460 u^2) v $ &	$ (1/18818)    (12903 - 44330 u + 35636 u^2) $ \\
$\langle \cdot, F_0 \mb{e}_{7}\rangle$ & $ (1/1358)    (245 - 656 u - 948 u^2) $ &	$ (4/679) (109 - 200 u + 125 u^2) v $ \\
$\langle \cdot, F_0 \mb{e}_{8}\rangle$ & $ -(4/461041)    (61800 - 34586 u + 543 u^2) v $ &	$ (1/   131726) (-41951 + 80776 u - 28660 u^2) $ \\
$\langle \cdot, F_0 \mb{e}_{9}\rangle$ & $ (2/679) (-102 + 342 u + 53 u^2) $ &	$ (6/4753)    (320 - 1033 u + 488 u^2) v $ \\
$\langle \cdot, F_0 \mb{e}_{10}\rangle$ & $ -(2/65863)    (13272 - 6927 u + 26312 u^2) v $ &	$ (30/65863)    (662 - 622 u + 25 u^2) $ \\
$\langle \cdot, F_0 \mb{e}_{11}\rangle$ & $ (3/922082)    (47265 - 111466 u + 50436 u^2) $ &	$ -(3/461041)    (68953 - 192732 u + 128460 u^2) v $ \\
$\langle \cdot, F_0 \mb{e}_{12}\rangle$ & $ (3/461041)    (68953 - 192732 u + 128460 u^2) v $ &	$ -(9/131726)    (11911 - 36462 u + 25820 u^2) $ \\
$\langle \cdot, F_0 \mb{e}_{13}\rangle$ & $ (1/922082) (-166195 + 68402 u + 95304 u^2) $ &	$ (3/   461041) (41199 - 45674 u + 5436 u^2)    v $ \\
$\langle \cdot, F_0 \mb{e}_{14}\rangle$ & $ -(1/461041) (-96429 -     46894 u + 182108 u^2) v $ &	$ (3/   131726) (-14829 - 8154 u + 26800 u^2) $ \\\hline
&$\mb{e}_7$&$\mb{e}_{8}$\\\hline
$\langle \cdot, F_0 \mb{e}_{1}\rangle$ & $ (1/   679) (-569 + 2564 u - 2136 u^2) $ &	$ (4/65863)    (33298 - 78729 u + 50406 u^2) v $ \\
$\langle \cdot, F_0 \mb{e}_{2}\rangle$ & $ (4/679) (31 - 258 u + 40 u^2) v $ &	$ -(4/679)    (129 - 113 u + 10 u^2) $ \\
$\langle \cdot, F_0 \mb{e}_{3}\rangle$ & $ (1/   9506) (-1347 + 2080 u + 13444 u^2) $ &	$ (2/   65863) (-7673 + 14458 u + 17178 u^2)    v $ \\
$\langle \cdot, F_0 \mb{e}_{4}\rangle$ & $ (2/4753) (-439 + 1432 u + 3082 u^2)    v $ &	$ (1/131726)    (66455 + 40832 u - 105836 u^2) $ \\
$\langle \cdot, F_0 \mb{e}_{5}\rangle$ & $ (1/1358)    (245 - 656 u - 948 u^2) $ &	$ -(4/461041) (61800 - 34586 u + 543 u^2)    v $ \\
$\langle \cdot, F_0 \mb{e}_{6}\rangle$ & $ (4/679) (109 - 200 u + 125 u^2)    v $ &	$ (1/131726) (-41951 + 80776 u - 28660 u^2) $ \\
$\langle \cdot, F_0 \mb{e}_{7}\rangle$ & $ (1/9506)    (4441 - 13812 u - 8488 u^2) $ &	$ -(6/679) (-27 - 229 u + 422 u^2)    v $ \\
$\langle \cdot, F_0 \mb{e}_{8}\rangle$ & $ -(6/679) (-27 - 229 u +     422 u^2) v $ &	$ (1/131726)    (120805 - 361636 u + 219960 u^2) $ \\
$\langle \cdot, F_0 \mb{e}_{9}\rangle$ & $ (1/9506)    (857 - 14736 u + 33848 u^2) $ &	$ -(2/4753) (1709 - 5171 u + 722 u^2)    v $ \\
$\langle \cdot, F_0 \mb{e}_{10}\rangle$ & $ (2/4753) (1709 - 5171 u + 722 u^2) v $ &	$ (1/   131726) (-103345 + 33000 u + 84232 u^2) $ \\
$\langle \cdot, F_0 \mb{e}_{11}\rangle$ & $ (2/   679) (-102 + 342 u + 53 u^2) $ &	$ (2/65863) (13272 - 6927 u + 26312 u^2)    v $ \\
$\langle \cdot, F_0 \mb{e}_{12}\rangle$ & $ -(6/4753) (320 - 1033 u + 488 u^2)    v $ &	$ (30/65863) (662 - 622 u + 25 u^2) $ \\
$\langle \cdot, F_0 \mb{e}_{13}\rangle$ & $ -(3/9506) (-587 + 2088 u + 3060 u^2) $ &	$ -(2/65863)    (7479 + 13866 u + 26278 u^2) v $ \\
$\langle \cdot, F_0 \mb{e}_{14}\rangle$ & $ (2/   4753) (-583 + 1116 u + 3134 u^2) v $ &	$ (1/131726)    (65465 + 43752 u - 88164 u^2) $

\end{tabular}}
\caption{\label{tabFF1}Elements of $F_0$ in the standard basis of $\mathbb{R}^{14}$. The elements of $F_i$
are obtained by a cyclic permutation of the indices by $2i$ (cont. next page).}
\end{table}

\addtocounter{table}{-1}

\begin{table}[p]
\resizebox{\textwidth}{!}{
\begin{tabular}{c|cc}
&$\mb{e}_9$&$\mb{e}_{10}$\\\hline
$\langle \cdot, F_0 \mb{e}_{1}\rangle$ & $ (1/   679) (-569 + 2564 u - 2136 u^2) $ &	$ -(4/65863)    (33298 - 78729 u + 50406 u^2) v $ \\
$\langle \cdot, F_0 \mb{e}_{2}\rangle$ & $ -(4/679) (31 - 258 u + 40 u^2)    v $ &	$ -(4/679) (129 - 113 u + 10 u^2) $ \\
$\langle \cdot, F_0 \mb{e}_{3}\rangle$ & $ -(3/9506) (-587 + 2088 u + 3060 u^2) $ &	$ (2/65863)    (7479 + 13866 u + 26278 u^2) v $ \\
$\langle \cdot, F_0 \mb{e}_{4}\rangle$ & $ -(2/4753) (-583 + 1116 u +     3134 u^2) v $ &	$ (1/131726) (65465 + 43752 u - 88164 u^2) $ \\
$\langle \cdot, F_0 \mb{e}_{5}\rangle$ & $ (2/679) (-102 + 342 u + 53 u^2) $ &	$ -(2/65863)    (13272 - 6927 u + 26312 u^2) v $ \\
$\langle \cdot, F_0 \mb{e}_{6}\rangle$ & $ (6/4753)    (320 - 1033 u + 488 u^2) v $ &	$ (30/65863)    (662 - 622 u + 25 u^2) $ \\
$\langle \cdot, F_0 \mb{e}_{7}\rangle$ & $ (1/9506) (857 - 14736 u + 33848 u^2) $ &	$ (2/4753)    (1709 - 5171 u + 722 u^2) v $ \\
$\langle \cdot, F_0 \mb{e}_{8}\rangle$ & $ -(2/4753)    (1709 - 5171 u + 722 u^2) v $ &	$ (1/   131726) (-103345 + 33000 u + 84232 u^2) $ \\
$\langle \cdot, F_0 \mb{e}_{9}\rangle$ & $ (1/9506) (4441 - 13812 u - 8488 u^2) $ &	$ (6/   679) (-27 - 229 u + 422 u^2) v $ \\
$\langle \cdot, F_0 \mb{e}_{10}\rangle$ & $ (6/   679) (-27 - 229 u + 422 u^2) v $ &	$ (1/131726)    (120805 - 361636 u + 219960 u^2) $ \\
$\langle \cdot, F_0 \mb{e}_{11}\rangle$ & $ (1/1358) (245 - 656 u - 948 u^2) $ &	$ (4/461041)    (61800 - 34586 u + 543 u^2) v $ \\
$\langle \cdot, F_0 \mb{e}_{12}\rangle$ & $ -(4/679)    (109 - 200 u + 125 u^2) v $ &	$ (1/   131726) (-41951 + 80776 u - 28660 u^2) $ \\
$\langle \cdot, F_0 \mb{e}_{13}\rangle$ & $ (1/   9506) (-1347 + 2080 u + 13444 u^2) $ &	$ -(2/65863) (-7673 + 14458 u +     17178 u^2) v $ \\
$\langle \cdot, F_0 \mb{e}_{14}\rangle$ & $ -(2/4753) (-439 + 1432 u + 3082 u^2)    v $ &	$ (1/131726) (66455 + 40832 u - 105836 u^2) $ \\\hline
&$\mb{e}_{11}$&$\mb{e}_{12}$\\\hline
$\langle \cdot, F_0 \mb{e}_{1}\rangle$ & $ (2/461041)    (19669 - 35394 u + 47802 u^2) $ &	$ (4/461041)    (88155 - 230075 u + 174752 u^2) v $ \\
$\langle \cdot, F_0 \mb{e}_{2}\rangle$ & $ 0 $ &	$ 0 $ \\
$\langle \cdot, F_0 \mb{e}_{3}\rangle$ & $ (1/   922082) (-166195 + 68402 u + 95304 u^2) $ &	$ -(3/461041)    (41199 - 45674 u + 5436 u^2) v $ \\
$\langle \cdot, F_0 \mb{e}_{4}\rangle$ & $ (1/   461041) (-96429 - 46894 u + 182108 u^2) v $ &	$ (3/   131726) (-14829 - 8154 u + 26800 u^2) $ \\
$\langle \cdot, F_0 \mb{e}_{5}\rangle$ & $ (3/922082)    (47265 - 111466 u + 50436 u^2) $ &	$ (3/461041)    (68953 - 192732 u + 128460 u^2) v $ \\
$\langle \cdot, F_0 \mb{e}_{6}\rangle$ & $ -(3/461041) (68953 - 192732 u + 128460 u^2)    v $ &	$ -(9/131726) (11911 - 36462 u + 25820 u^2) $ \\
$\langle \cdot, F_0 \mb{e}_{7}\rangle$ & $ (2/   679) (-102 + 342 u + 53 u^2) $ &	$ -(6/4753) (320 - 1033 u + 488 u^2)    v $ \\
$\langle \cdot, F_0 \mb{e}_{8}\rangle$ & $ (2/65863)    (13272 - 6927 u + 26312 u^2) v $ &	$ (30/65863)    (662 - 622 u + 25 u^2) $ \\
$\langle \cdot, F_0 \mb{e}_{9}\rangle$ & $ (1/1358)    (245 - 656 u - 948 u^2) $ &	$ -(4/679) (109 - 200 u + 125 u^2) v $ \\
$\langle \cdot, F_0 \mb{e}_{10}\rangle$ & $ (4/461041)    (61800 - 34586 u + 543 u^2) v $ &	$ (1/   131726) (-41951 + 80776 u - 28660 u^2) $ \\
$\langle \cdot, F_0 \mb{e}_{11}\rangle$ & $ (1/922082)    (213769 - 567974 u - 163748 u^2) $ &	$ (3/461041)    (68953 - 192732 u + 128460 u^2) v $ \\
$\langle \cdot, F_0 \mb{e}_{12}\rangle$ & $ (3/461041)    (68953 - 192732 u + 128460 u^2) v $ &	$ (1/18818)    (12903 - 44330 u + 35636 u^2) $ \\
$\langle \cdot, F_0 \mb{e}_{13}\rangle$ & $ (1/461041)    (759 + 107533 u + 100614 u^2) $ &	$ -(1/461041) (-149167 - 245274 u +     634392 u^2) v $ \\
$\langle \cdot, F_0 \mb{e}_{14}\rangle$ & $ -(3/461041)    (8985 + 89258 u + 119032 u^2) v $ &	$ (1/65863)    (20449 + 1367 u - 33410 u^2) $ \\\hline
&$\mb{e}_{13}$&$\mb{e}_{14}$\\\hline
$\langle \cdot, F_0 \mb{e}_{1}\rangle$ & $ -(2/461041)    (72529 - 63570 u + 30298 u^2) $ &	$ -(4/461041) (-60408 - 31391 u +     115516 u^2) v $ \\
$\langle \cdot, F_0 \mb{e}_{2}\rangle$ & $ (4/   49) (-4 - u + 10 u^2) v $ &	$ (1/7) $ \\
$\langle \cdot, F_0 \mb{e}_{3}\rangle$ & $ (1/461041)    (28120 + 29493 u - 21622 u^2) $ &	$ (1/   461041) (-79117 - 84650 u + 70536 u^2) v $ \\
$\langle \cdot, F_0 \mb{e}_{4}\rangle$ & $ -(1/461041) (-79117 -     84650 u + 70536 u^2) v $ &	$ (1/   65863) (-21996 - 12961 u + 35782 u^2) $ \\
$\langle \cdot, F_0 \mb{e}_{5}\rangle$ & $ (1/   922082) (-166195 + 68402 u + 95304 u^2) $ &	$ -(1/461041) (-96429 -     46894 u + 182108 u^2) v $ \\
$\langle \cdot, F_0 \mb{e}_{6}\rangle$ & $ (3/461041)    (41199 - 45674 u + 5436 u^2) v $ &	$ (3/   131726) (-14829 - 8154 u + 26800 u^2) $ \\
$\langle \cdot, F_0 \mb{e}_{7}\rangle$ & $ -(3/9506) (-587 + 2088 u + 3060 u^2) $ &	$ (2/   4753) (-583 + 1116 u + 3134 u^2) v $ \\
$\langle \cdot, F_0 \mb{e}_{8}\rangle$ & $ -(2/65863) (7479 + 13866 u + 26278 u^2)    v $ &	$ (1/131726) (65465 + 43752 u - 88164 u^2) $ \\
$\langle \cdot, F_0 \mb{e}_{9}\rangle$ & $ (1/   9506) (-1347 + 2080 u + 13444 u^2) $ &	$ -(2/4753) (-439 + 1432 u +     3082 u^2) v $ \\
$\langle \cdot, F_0 \mb{e}_{10}\rangle$ & $ -(2/65863) (-7673 +     14458 u + 17178 u^2) v $ &	$ (1/131726)    (66455 + 40832 u - 105836 u^2) $ \\
$\langle \cdot, F_0 \mb{e}_{11}\rangle$ & $ (1/461041)    (759 + 107533 u + 100614 u^2) $ &	$ -(3/461041)    (8985 + 89258 u + 119032 u^2) v $ \\
$\langle \cdot, F_0 \mb{e}_{12}\rangle$ & $ -(1/461041) (-149167 - 245274 u +     634392 u^2) v $ &	$ (1/65863)    (20449 + 1367 u - 33410 u^2) $ \\
$\langle \cdot, F_0 \mb{e}_{13}\rangle$ & $ (1/461041) (-20089 - 45197 u - 24338 u^2) $ &	$ (1/   461041) (15943 + 262998 u + 380160 u^2)    v $ \\
$\langle \cdot, F_0 \mb{e}_{14}\rangle$ & $ (1/461041)    (15943 + 262998 u + 380160 u^2) v $ &	$ (1/   65863) (-23109 - 25645 u + 18926 u^2) $

\end{tabular}}
\caption{(cont.) Elements of $F_0$ in the standard basis of $\mathbb{R}^{14}$. The elements of $F_i$
are obtained by a cyclic permutation of the indices by $2i$.}
\end{table}


\begin{table}[p]
\resizebox{\textwidth}{!}{
\begin{tabular}{rll}
$\langle\mb{f}_0',\mb{e}_i\rangle = $ & $\tfrac{2}{679}(128 - 743 u + 816 u^2)$ & for $i=1$ \\
$\langle\mb{f}_0',\mb{e}_i\rangle = $ & $0$ & for $i=2,4,6,8,10,12,14$ \\
$\langle\mb{f}_0',\mb{e}_i\rangle = $ & $\tfrac{2}{7}(-4 - 3 u + 8 u^2)$ & for $i=3,5,7,9,11,13$ \\
$\langle\mb{f}_0',\mb{e}_i\rangle = $ & $\tfrac{2}{7}(-2 - u + 4 u^2)v$ & for $i=15,18,22,23,26,27,30,31,34,35,39,42$ \\
$\langle\mb{f}_0',\mb{e}_i\rangle = $ & $\tfrac{2}{7}(-5 - 3 u + 10 u^2)v$ & for $i=16,17,20,21,25,28,29,32,36,37,40,41$ \\
$\langle\mb{f}_0',\mb{e}_i\rangle = $ & $\tfrac{2}{679}(-618 - 273 u + 692 u^2)v$ & for $i=19,38$ \\
$\langle\mb{f}_0',\mb{e}_i\rangle = $ & $\tfrac{2}{679}(-649 - 1179 u + 2010 u^2)v$ & for $i=24,33$ \\
$\langle\mb{g}_1',\mb{e}_1\rangle = $ & $\langle\mb{g}_8',\mb{e}_1\rangle = $ & $-1-u$ \\
$\langle\mb{g}_1',\mb{e}_2\rangle = $ & $-\langle\mb{g}_8',\mb{e}_2\rangle = $ & $\tfrac{1}{8} (-15 - 2 u + 20 u^2)$ \\
$\langle\mb{g}_1',\mb{e}_{28}\rangle = $ & $\langle\mb{g}_8',\mb{e}_{29}\rangle = $ & $-2(1+u)v$ \\
$\langle\mb{g}_{15}',\mb{e}_1\rangle = $ & $\langle\mb{g}_{22}',\mb{e}_1\rangle = $ & $\tfrac{3}{2}(-3+4u^2)v$ \\
$\langle\mb{g}_{15}',\mb{e}_{27}\rangle = $ & $\langle\mb{g}_{22}',\mb{e}_{30}\rangle = $ & $\tfrac{1}{4}(-15-2u+20u^2)$ \\
$\langle\mb{g}_{15}',\mb{e}_{28}\rangle = $ & $\langle\mb{g}_{22}',\mb{e}_{29}\rangle = $ & $\tfrac{1}{2}(-9-2u+14u^2)$ \\
$\langle\mb{g}_{29}',\mb{e}_{27}\rangle = $ & $\langle\mb{g}_{36}',\mb{e}_{30}\rangle = $ & $\tfrac{1}{4}(-7-2u+12u^2)$ \\
$\langle\mb{g}_{29}',\mb{e}_{28}\rangle = $ & $\langle\mb{g}_{36}',\mb{e}_{29}\rangle = $ & $\tfrac{1}{2}(-13-2u+18u^2)$ 

\end{tabular}}
\caption{\label{tabfgp}Elements of $\mb{f}_0'$, $\mb{g}_1'$, $\mb{g}_8'$, $\mb{g}_{15}'$, $\mb{g}_{22}'$, $\mb{g}_{29}'$,
and $\mb{g}_{36}'$ in the
standard basis of $\mathbb{R}^{14}$. Elements not given explicitly are zero. The elements of $\mb{f}_i'$ and $\mb{g}_i'$ for other
values of $i$ are obtained by appropriate permutation of the indices (for example, to obtain $\mb{g}_{i+i'}'$ from $\mb{g}_i'$,
cycle the first 14 coordinates by $2i'$ and the last 28 by $4i'$).}
\end{table}

\end{document}